\documentclass{amsart}
\usepackage{amsmath,amsthm,amssymb,amsfonts}
\input xy
\xyoption{all}
\usepackage[active]{srcltx}

\theoremstyle{plain}
\newtheorem{thm}{Theorem}[section]
\newtheorem{prop}[thm]{Proposition}
\newtheorem{lem}[thm]{Lemma}
\newtheorem{cor}[thm]{Corollary}

\theoremstyle{definition}
\newtheorem{defn}[thm]{Definition}

\theoremstyle{remark}
\newtheorem{rem}[thm]{Remark}

\newtheorem{exam}[thm]{Example}

\newcommand{\Hom}{\ensuremath{\mathrm{Hom}} }

\newcommand{\Aut}{\ensuremath{\mathrm{Aut}} }

\newcommand{\Out}{\ensuremath{\mathrm{Out}} }

\newcommand{\id}{\ensuremath{\mathrm{id}} }

\newcommand{\rank}{\ensuremath{\mathrm{rk }} }

\newcommand{\normal}{\ensuremath{\vartriangleleft} }
\newcommand{\normaleq}{\ensuremath{\trianglelefteq} }

\newcommand{\nnormal}{\ensuremath{\ntriangleleft} }

\newcommand{\F}{\ensuremath{\mathcal{F}} }
\newcommand{\G}{\ensuremath{\mathcal{G}} }

\renewcommand{\F}{\ensuremath{\mathcal{F}} }
\newcommand{\CG}{\ensuremath{\mathcal{G}} }

\newcommand{\Syl}{\ensuremath{\mathrm{Syl}} }
\newcommand{\tr}{\ensuremath{\mathrm{tr}} }

\def\C{\ensuremath{\mathrm{C}} }
\def\J{\ensuremath{\mathrm{J}} }
\def\N{\ensuremath{\mathrm{N}} }
\def\O{\ensuremath{\mathrm{O}} }
\def\W{\ensuremath{\mathrm{W}} }
\def\Z{\ensuremath{\mathrm{Z}} }

\newcommand{\lar}[1]{\ar @{-}[#1]}
\newcommand{\larud}[1]{\ar @{-}@/^1pc/[#1]\ar @{-}@/_1pc/[#1]}
\newcommand{\darud}[1]{\ar @{--}@/^1pc/[#1]\ar @{--}@/_1pc/[#1]}

\begin{document}
\bibliographystyle{plain}

\title{Sparse Fusion Systems}
\author{Adam Glesser}
\date{\today}
\begin{abstract}
We define sparse saturated fusion systems and show that, for odd primes, sparse systems are constrained. This simplifies the proof of the Glauberman-Thompson $p$-nilpotency theorem for fusion systems and a related theorem of Stellmacher. We then define a more restrictive class of saturated fusion systems, called extremely sparse, that are constrained for all primes.
\end{abstract}
\maketitle

\section{Introduction}
For those concerned with fusion in finite groups, a $p$-nilpotent group is as trivial as it gets. As such, $p$-nilpotency criteria are inherently interesting. Following up his dissertation, John Thompson \cite{Thompson1964} proved, for odd primes $p$, a $p$-nilpotency criterion for finite groups that reduces the problem to a $p$-local question, namely, checking the fusion in two subgroups, $\N_G(\J(P))$ and $\C_G(\Z(P))$. Recently, D\'iaz, Mazza, Park and the author proved a generalization of this theorem for fusion systems (see \cite{DiazGlesserMazzaParkGT}). 

In \cite{KessarLinckelmann2008}, Kessar and Linckelmann state and prove a generalization to fusion systems of Glauberman's improved version of Thompson's $p$-nilpotency result, one that reduces the question of $p$-nilpotency in $G$ to that of the $p$-nilpotency of $
\N_G(\Z(\J(P)))$; motivated, presumably, by Gorenstein's treatment in \cite{Gorenstein1980}, this result is referred to as the \textit{Glauberman---Thompson $p$-nilpotency theorem}. 

In the present work, we aim to shorten the proof of this last result, showing that it is a consequence of the fusion system version of Thompson's $p$-nilpotency criterion and Glauberman's $p$-nilpotency criterion for groups. This simplification follows from the following observation: minimal counterexamples to statements	 whose conclusion is that a fusion system is trivial tend to only have the trivial subsystem as a proper subsystem on the same $p$-group. We call such a fusion system \textit{sparse}. In Section \ref{S:sparse fusion systems}, we prove that, for $p$ odd, any sparse fusion system is constrained, i.e., it contains a normal centric subgroup; this implies that the system comes from a finite group and, thus, is subject to fusion results for finite groups. For $p = 2$, this result does not hold in general and we will give an example of a fusion system exhibiting this deficiency.

To demonstrate the ubiquity of sparse fusion systems, we present several further examples. In Section \ref{ss:Navarro}, we generalize a result of Navarro to fusion systems. This result holds for all primes and  strengthens Navarro's original result. Furthermore, when applied to the situation of $p$-blocks, we obtain a new nilpotency criterion for blocks generalizing the classical result (see \cite{Kulshammer1980} and \cite{BrouePuig1980-2}) that a block with inertial index 1 and abelian defect group is nilpotent. In this section we also generalize a recent result (\cite{Weigel2009pre}) of Weigel on \textit{slim} $p$-groups. Finally, in Section \ref{S:extremely sparse fusion systems}, we consider saturated fusion systems for which a proper subsystem on \textit{any} $p$-subgroup is trivial. These \textit{extremely sparse} fusion systems are always constrained (even for $p =2$) and we give a classification of these systems along with an example of their use.

The genesis of this paper is a question posed by Radha Kessar asking for meta-theorems that decide when a result from group theory will hold for fusion systems. Consider this a very tiny first step in that direction.

\section{Preliminaries}
We begin with a brief summary of saturated fusion systems. The concept of a (saturated) fusion system is originally due to Puig (\cite{PuigUnpublished}) and the approach used in this paper is the one adopted by Broto, Levi and Oliver (\cite{BrotoLeviOliver2003}). For more detail, proper motivation or the proofs of any theorems given without justification, we refer the reader to \cite{Linckelmann2007}. 
\subsection*{Saturated Fusion Systems} Let $p$ be a prime, $P$ a finite $p$-group, $\F$ a category whose objects are the subgroups of $P$ and for $Q,R \leq P$, $\Hom_\F(Q,R)$ is a subset of the injective group homomorphisms from $Q$ to $R$. Composition of morphisms is given as the usual composition of group homomorphisms. Denote the $\F$-isomorphism class of $Q$ by $Q^\F$. 

\begin{enumerate}
\item A subgroup $Q \leq P$ is \textit{fully $\F$-normalized} (respectively, \textit{fully $\F$-centralized}) if $|\N_P(Q)| \geq |\N_P(R)|$ (respectively, $|\C_P(Q)| \geq |\C_P(R)|$) for all $R \in Q^\F$.

\item The category $\F$ is a \textit{saturated fusion system} on $P$ if the following hold for all subgroups $Q,R \leq P$:
\begin{enumerate}
\item If $Q \leq R$, then the inclusion map from $Q$ to $R$ is a morphism in $\F$.
\item If $\phi \in \Hom_\F(Q,R)$, then the induced isomorphism from $Q$ to $\phi(Q)$ and its inverse are also in $\F$.
\item $\Hom_P(Q,R) \subseteq \Hom_\F(Q,R)$ where the former set denotes the group homomorphisms from $Q$ to $R$ induced by conjugation with an element of $P$.
\item\label{d:Sylow axiom}(Sylow axiom) $\Aut_P(P)$ is a Sylow $p$-subgroup of $\Aut_\F(P)$
\item\label{d:extension axiom}(Extension axion) If $\phi \in \Hom_\F(Q,P)$ such that $\phi(Q)$ is fully $\F$-normalized, then $\phi$ extends to a morphism in $\Hom_\F(\N_\phi, P)$. Here, $\N_\phi$ denotes the inverse image in $\N_P(Q)$ of $\Aut_P(Q) \cap (\phi^{-1} \Aut_P(\phi(Q)) \phi)$.
\end{enumerate}
\end{enumerate}

\begin{exam}
Let $G$ be a finite group with Sylow $p$-subgroup $P$. For $g \in G$, let $c_g$ denote the automorphism of $G$ given by conjugation by $g$ and for $Q,R \leq P$, set
\[
\Hom_\F(Q,R) = \{c_g|_Q \ |\ g \in G, Q^g \leq R\}
\]
It is a straight forward application of the Sylow theorems to show that this gives a saturated fusion system on $P$  and we denote it by $\F_P(G)$. Recall that a group $G$ is $p$-\textit{nilpotent} if it has a normal $p$-complement, i.e, if $G = P O_{p'}(G)$. In the language of fusion systems, $G$ is $p$-nilpotent if and only if $\F_P(G) = \F_P(P)$. In general, we call a saturated fusion system $\F$ on $P$ \textit{trivial} if $\F = \F_P(P)$.
\end{exam}

Part of the motivation for the above definition is that it allows us to mimic quite a bit of local group theory, including normalizers, centralizers and Alperin's fusion theorem. 

\subsection*{Subsystems and Quotient Systems.} Let $\F$ be a saturated fusion system on a finite $p$-group P. A subcategory $\G$ of $\F$ is a a \textit{saturated subsystem} of $\F$ if there exists a subgroup $Q$ of $P$ such that $\G$ is a saturated fusion system on $Q$. Puig defined subcategories, corresponding to local subgroups in finite group theory, $\N_\F(Q)$, $\N_P(Q)\C_\F(Q)$ and $\C_\F(Q)$ whose objects are the subgroups of $\N_P(Q)$, $\N_P(Q)$, and $\C_P(Q)$, respectively and where for subgroups $R$ and $S$ of these respective groups,
\begin{align*}
\Hom_{\N_\F(Q)}(R,S) & = \{\phi \in \Hom_\F(R,S) \ |\ \exists\ \varphi \in \Hom_\F(QR,QS):\ \varphi|_R = \phi\} \\
\Hom_{\N_P(Q)\C_\F(Q)}(R,S) & = \{\phi \in \Hom_\F(R,S) \ |\ \exists\ \varphi \in \Hom_\F(QR,QS):\ \varphi|_R = \phi, \varphi|_Q \in \Aut_P(Q)\} \\
\Hom_{\C_\F(Q)}(R,S) & = \{\phi \in \Hom_\F(R,S) \ |\ \exists\ \varphi \in \Hom_\F(QR,QS):\ \varphi|_R = \phi, \varphi|_Q = \id_Q\}
\end{align*}
If $Q$ is fully $\F$-centralized, then $\N_P(Q)\C_\F(Q)$ and $\C_\F(Q)$ are saturated subsystems of $\F$. If $Q$ is fully $\F$-normalized, then $\N_\F(Q)$ is saturated. In the special case where $\N_\F(Q) = \F$, we say that $Q$ is \textit{normal} in $\F$ and write $Q \normal \F$. If $\C_\F(Q) = \F$, we say that $Q$ is \textit{central} in $\F$. When $Q$ is fully $\F$-normalized, all of the above systems are saturated and we get the following chain of saturated subsystems of $\F$
\[
\C_\F(Q) \subseteq \N_P(Q)\C_\F(Q) \subseteq \N_\F(Q) \subseteq \F.
\]
The largest normal and central subgroups of $\F$ are denoted by $\O_p(\F)$ and $\Z(\F)$, respectively. For more details and proofs, see \cite[\S 3]{Linckelmann2007}.

It may happen that a subgroup $Q$ is normal in $P$, but not normal in $\F$. In this context, there are a couple of gradations worth mentioning. If $Q$ is stabilized by every $\F$-morphism defined on $Q$, then $Q$ is called \textit{weakly $\F$-closed}. Furthermore, if the image, under any $\F$-morphism, of every subgroup of $Q$ remains in $Q$, then $Q$ is \textit{strongly $\F$-closed}. Therefore, if $Q \leq P$, then  
\[
Q \normal \F \Longrightarrow \text{$Q$ strongly $\F$-closed} \Longrightarrow \text{$Q$ weakly $\F$-closed} \Longrightarrow Q \normaleq P
\] 

When $Q$ is strongly $\F$-closed, Puig defined a category $\F/Q$ whose objects are the subgroups of $P/Q$ and whose morphisms are induced from $\F$ and stablize $Q$ and proved that it is a saturated fusion system. We omit a precise definition here as we will only need them in the context of the following proposition of Kessar and Linckelmann. For more details on $\F/Q$, we recommend the recent article \cite{Craven2010} of David Craven where several technical flaws in earlier treatments are overcome and $\F/Q$ is proven to be a saturated fusion system. 	
\\

Our goal in many cases is to reduce to the case where $\F = P\C_\F(Q)$ for some $Q \normaleq P$. To do so, we use the following result of Kessar and Linckelmann.

\begin{prop}(\cite[Proposition 3.4]{KessarLinckelmann2008})\label{KL3.4}
Let $\CG \subseteq \F$ be saturated fusion systems on a finite $p$-group $P$. If $Q$ and $R$ are normal subgroups of $P$ such that $Q \leq R$ and $\F = P\C_\F(Q)$, then 
\[
\text{$\CG = \N_\F(R)$ if and only if $\CG/Q = \N_{\F/Q}(R/Q)$}.
\]
\end{prop}


\subsection*{Alperin's Fusion Theorem}
The theorem we refer to here as Alperin's fusion theorem is a generalization to fusion systems of a theorem first proved by Alperin and improved upon by Goldschmidt and Puig. This version utilizes $\F$-essential subgroups, a class of subgroups that is, in some sense, minimal when it comes to generating fusion. Recall that a proper subgroup $H$ of a finite group $G$ is \textit{strongly $p$-embedded} if it contains a nontrivial Sylow $p$-subgroup $P$ of $G$ and $H \cap P^x = 1$ for any $x \in G \setminus H$. In particular, if $\O_p(G) \neq 1$, then $G$ has no strongly $p$-embedded subgroup. 

\begin{defn}
Let $\F$ be a saturated fusion system on a finite $p$-group $P$ and let $Q$ be a subgroup of $P$.
\begin{enumerate}
\item $Q$ is $\F$-\textit{centric} if $\C_P(R) \leq R$ for all $R \in Q^\F$.
\item $Q$ is $\F$-\textit{essential} if $Q$ is $\F$-centric and $\Out_\F(Q) = \Aut_\F(Q)/\Aut_Q(Q)$ has a strongly $p$-embedded subgroup. 
\end{enumerate}
\end{defn}

We start with a few useful trivialities and then a lesser known property of $\F$-essential subgroups and the extension axiom. 

\begin{lem}\label{l:essential lemma}
Let $\F$ be a saturated fusion system on a finite $p$-group $P$. 
\begin{enumerate}
\item\label{lp:P not essential} $P$ is not $\F$-essential.
\end{enumerate}
If $Q < P$ is $\F$-essential, then
\begin{enumerate}
\setcounter{enumi}{1}
\item\label{lp:p'-autos for essentials} $\Aut_\F(Q)$ is not a $p$-group. 
\item\label{lp:nocyclicessentials} $\Aut_\F(Q)$ does not have a normal Sylow $p$-subgroup. In particular, a cyclic $p$-group is never essential.
\item\label{lp:essential no extend} If $R \in Q^\F$ is a fully $\F$-normalized subgroup of $P$, then there exists $\varphi \in \Hom_\F(Q,R)$ such that $\N_\varphi = Q$.
\end{enumerate}
 
\end{lem}

\begin{proof}
The Sylow axiom in the definition of a saturated fusion system implies that $\Out_\F(P)$ is a $p'$-group and so, by definition, it cannot have a strongly $p$-embedded subgroup. Thus, $P$ is not $\F$-essential, proving (\ref{lp:P not essential}). By definition, a $p$-group cannot have a strongly $p$-embedded subgroup and so $\Out_\F(Q)$ is not a $p$-group. This implies the existence of a $p'$-automorphism for $Q$ in $\F$ proving (\ref{lp:p'-autos for essentials}). If $\Aut_\F(Q)$ has a normal Sylow $p$-subgroup, then (as it must contain $\Aut_P(Q)$) $\O_p(\Out_\F(Q)) \neq 1$; this is a contraction since it implies $\Out_\F(Q)$ has no strongly $p$-embedded subgroup. To prove (\ref{lp:essential no extend}), note that as $Q$ is $\F$-essential, we may choose $A, B \in \Syl_p(\Aut_\F(Q))$ such that $A \cap B = \Aut_Q(Q)$. Furthermore, without loss of generality, we assume that $\Aut_P(Q) \leq A$. Since $R$ is fully $\F$-normalized, $\Aut_P(R)$ is a Sylow $p$-subgroup of $\Aut_\F(R) \cong \Aut_\F(Q)$ and so if $\phi \in \Hom_\F(Q,R)$, then $\phi^{-1}\Aut_P(R)\phi \in \Syl_p(\Aut_\F(Q))$. Thus, there exists $\psi \in \Aut_\F(Q)$ such that $\psi^{-1}\phi^{-1}\Aut_P(R)\phi\psi = B$. Set $\varphi = \phi \psi$. 
If $c : \N_P(Q) \to \Aut_P(Q)$ denotes the homomorphism sending an element $u \in \N_P(Q)$ to the automorphism $c_u$ given by conjugation by $u$, then 
\[
\N_\varphi = c^{-1}(\Aut_P(Q) \cap \varphi^{-1}\Aut_P(R)\varphi ) \leq c^{-1}(A \cap B) = c^{-1}(\Aut_Q(Q)) = Q\C_P(Q) = Q
\]
where the last equality follows since $Q$ is $\F$-centric.
\end{proof}

The author wishes to thank Radu Stancu for suggesting part (\ref{lp:essential no extend}) of the above lemma and David Craven for his help in developing the above proof. Another important property of $\F$-essential subgroups is that they always contain $\O_p(\F)$.

\begin{prop}\cite[Proposition 1.6]{BrotoCastellanaGrodalLeviOliver2005}\label{p:O_p(F) in essentials}
Let $\F$ be a saturated fusion system on a finite $p$-group $P$. If $Q$ is an $\F$-essential subgroup of $P$, then $\O_p(\F) \leq Q$.
\end{prop}

\noindent It is worth pointing out that, for nontrivial $P$, every $\F$-centric subgroup properly contains $\Z(P)$ so that, in fact, every $\F$-essential subgroup contains $\Z(P)\O_p(\F)$.

We now state Alperin's fusion theorem. Morally, it tells us that a saturated fusion system is determined by the $\F$-automorphisms of $P$ and the $\F$-essential subgroups of $P$.

\begin{thm}[Alperin's fusion theorem]\label{t:AFT} 
Let $\F$ be a saturated fusion system on a finite $p$-group $P$. Every $\F$-isomorphism is the composition of finitely many morphisms of the form $\phi : Q \to R$ where there exists $Q, R \leq S \leq P$ such that $S = P$ or $S$ is $\F$-essential and there exists $\alpha \in \Aut_\F(S)$ such that $\alpha|_Q = \phi$.
 
\end{thm}

By \cite[Proposition 2.10]{DiazGlesserMazzaParkGT}, a saturated fusion system $\F$ is generated by the $\F$-automorphisms of a set of representatives of the $\F$-isomorphism classes of $\F$-essential subgroups of $P$. This motivates the following definition.

\begin{defn}
Let $\F$ be a saturated fusion system on a finite $p$-group $P$. The \textit{essential rank} of $\F$ is the number of $\F$-isomorphism classes of $\F$-essential subgroups of $P$. The essential rank of $\F$ is denoted by $\rank_e(\F)$.
\end{defn}

The structure of a saturated fusion system $\F$ on a finite $p$-group $P$ with essential rank 0 is particularly straight forward. 

\begin{lem}\label{l:no essentials}
If $\F$ is a saturated fusion system on a finite $p$-group $P$, then the following are equivalent.
\begin{enumerate}
\item $P \normal \F$\label{lp:P <| F}
\item $\rank_e(\F) = 0$.\label{lp:rk_e(F)=0}
\item $\F = \F_P(P \rtimes \Out_\F(P))$.\label{lp:F=P.Out_F(P)}
\end{enumerate}
\end{lem}

\begin{proof}
If $P \normal \F$, then $\O_p(\F) = P$ and so, by Lemma \ref{l:essential lemma} and Proposition \ref{p:O_p(F) in essentials}, there are no $\F$-essential subgroups of $P$. Conversely, if there are no $\F$-essentials, then Alperin's fusion theorem implies that each $\F$-automorphism extends to $P$ and $P$ is normal in $\F$. This shows that (\ref{lp:P <| F}) is equivalent to (\ref{lp:rk_e(F)=0}). As $P$ is clearly normal in $\F_P(P \rtimes \Out_\F(P))$, it remains to show that (\ref{lp:P <| F}) implies (\ref{lp:F=P.Out_F(P)}). In this case, by (\ref{lp:rk_e(F)=0}) and Alperin's fusion theorem, $\F$ is generated by $\Aut_\F(P)$. The result now follows since $\Aut_P(P) \in \Syl_p(\Aut_\F(P))$. 

\end{proof}

\subsection*{Constrained Fusion Systems}
In \cite[Proposition 4.3]{BrotoCastellanaGrodalLeviOliver2005}, Broto, Castellana, Grodal, Levi and Oliver prove that if $\F$ is a saturated fusion system, then for every fully $\F$-normalized $\F$-centric subgroup $Q$ of $P$, there exists a unique finite group $G$ (up to isomorphism) such that $\O_{p'}(G) = 1$, $\C_{G}(\O_p(G)) \leq \O_p(G)$ and such that $\N_\F(Q) = \F_{\N_P(Q)}(G)$. In this case there is an exact sequence
\[
1 \to \Z(Q) \to G \to \Aut_\F(Q) \to 1
\] In particular, if $Q$ is normal in $\F$, then $\F = \F_P(G)$. In this case, where $\F$ has a normal $\F$-centric subgroup, $\F$ is called \emph{constrained}. By Lemma \ref{l:no essentials}, any saturated fusion system with essential rank 0 is constrained. The importance of being constrained is that it reduces some questions about fusion systems to questions about groups (see  \cite{DiazGlesserMazzaParkTr}, \cite{DiazGlesserMazzaParkGT} or \cite{KessarLinckelmann2008} for some recent examples).

	\section{Sparse Fusion Systems}\label{S:sparse fusion systems}

We begin by proving a useful lemma due to Onofrei and Stancu. This will help us reduce to the case where the fusion system is of the form $\F = P\C_\F(Q)$ for some $Q \normal \F$.
 
\begin{lem}\cite[Lemma 3.7]{OnofreiStancu2009}\label{Stancu}
 Let $\F$ be a fusion system on a finite $p$-group $P$, and let $Q \leq P$. If $Q \normal \F$, then
\[
 \F = \langle P\C_\F(Q), \N_\F(Q\C_P(Q)) \rangle.
\]

\end{lem}

\begin{proof}
 Let $T$ be an $\F$-essential subgroup of $P$ and take $\varphi \in \Aut_\F(T)$.  As $Q$ is weakly $\F$-closed, we have $\theta = \varphi|_Q \in \Aut_\F(Q)$ and since $TQ\C_P(Q) \leq \N_\theta$, there is $\psi \in \Hom_\F(TQ\C_P(Q),P)$ such that $\psi|_Q = \varphi|_Q$.  Then
\[
    \varphi = (\varphi \circ (\psi|_T)^{-1}) \circ \psi|_T;
\]
$\varphi \circ (\psi|_T)^{-1}$ is a morphism in $P\C_\F(Q)$; $\psi|_T$ is a morphism in $\N_\F(Q\C_P(Q))$ because $\psi(Q\C_P(Q)) = Q\C_P(Q)$.  Thus $\varphi$ is a morphism in $\langle P\C_\F(Q), \N_\F(Q\C_P(Q)) \rangle$.  By Alperin's fusion theorem, it follows that $\F = \langle P\C_\F(Q), \N_\F(Q\C_P(Q)) \rangle$.
\end{proof}

\begin{prop}\label{F=PC_F(Q)2}
Let $\F$ be a saturated fusion system on a finite $p$-group $P$. If $Q$ and $R$ are normal subgroups of $P$ such that $Q \leq R$ and $\F = P\C_\F(Q)$, then 
\begin{enumerate}
\item $\N_\F(R)$ is trivial if and only if $\N_{\F/Q}(R/Q)$ is trivial. In particular, $\F$ is trivial if and only if $\F/Q$ is trivial.

\item $R \normal \F$ if and only if $R/Q \normal \F/Q$.
\end{enumerate}
\end{prop}

\begin{proof}
We obtain (1) by applying Proposition \ref{KL3.4} with $\CG = \F_P(P)$. Statement (2) is obtained when applying Proposition \ref{KL3.4} with $\CG = \F$.
\end{proof}

\begin{defn}
 A nontrivial saturated fusion system $\F$ on a finite $p$-group $P$ is called {\em sparse} if the only proper subfusion system of $\F$ on $P$ is $\F_P(P)$.
\end{defn}

The motivation for defining this class of fusion systems is that often a (putative) minimal counterexample to a theorem whose conclusion is that a saturated fusion system is trivial will be a sparse
  fusion system.  A tangible example of a sparse fusion system is the fusion system of $S_4$ on
   $D_8$. In the picture below, the fusion of $D_8$ on $D_8$ is
   described by the circled dots, while the additional fusion of
   $S_4$ on $D_8$ is described by the dashed line.
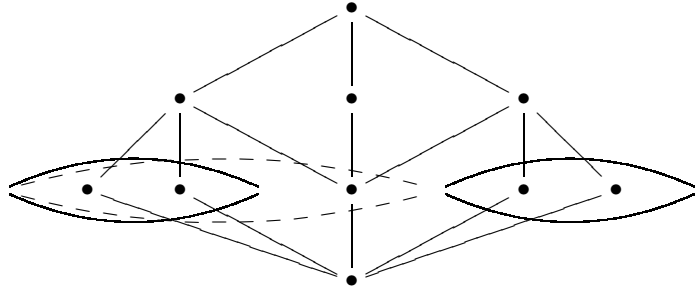
\[
\xymatrix{
    & & & & \bullet \lar{dll} \lar{d} \lar{drr} & & & \\
    & & \bullet \lar{dl} \lar{d} \lar{drr} & & \bullet \lar{d} & &
        \bullet \lar{dll} \lar{d} \lar{dr} & \\
    \darud{rrrrr} \larud{rrr} & \bullet \lar{drrr} & \bullet
        \lar{drr} & & \bullet \lar{d} & \larud{rrr} & \bullet
        \lar{dll} & \bullet \lar{dlll} & \\
    & & & & \bullet & & &}
\]
Sparse fusion systems are necessarily ubiquitous objects. Take any nontrivial saturated fusion system $\F$ on a finite $p$-group $P$ and consider the lattice of subsystems of $\F$ on $P$. Any minimal nontrivial subsystem in this lattice is a sparse fusion system. In particular, for any finite $p$-group on which there is at least one nontrivial saturated fusion system, there is a sparse fusion system on that $p$-group.


For a finite $p$-group $P$, the Thompson subgroup, $\J(P)$, is the subgroup of $P$ generated by the abelian subgroups of $P$ of maximal order. A classical result of Thompson is that a group $G$ with Sylow $p$-subgroup $P$ is $p$-nilpotent if and only if $\N_G(\J(P))$ and $\C_G(\Z(P))$ are $p$-nilpotent. This was recently extended to saturated fusion systems in \cite{DiazGlesserMazzaParkGT}. Recall from \cite{KessarLinckelmann2008} that a saturated fusion system is $S_4$-free if all of the groups arising from the normalizers of $\F$-centric, fully $\F$-normalized subgroups of $P$, as in \cite[Proposition 4.3]{BrotoCastellanaGrodalLeviOliver2005}, are $S_4$-free.
	
\begin{thm}\cite[Theorem 4.5]{DiazGlesserMazzaParkGT}
Let $\F$ be a saturated fusion system on a finite $p$-group $P$ where $p$ is odd or $\F$ is $S_4$-free. If $\N_\F(\J(P)) = \F_P(P) = \C_\F(\Z(P))$, then $\F = \F_P(P)$.
\end{thm}

\noindent We use this result to detect constraint in sparse fusion systems.

\begin{thm}\label{t:sparse}
Let $\F$ be a sparse fusion system on a finite $p$-group $P$.
\begin{enumerate}

 \item Let $Q \leq P$ such that $Q \normal \F$. If $Q\C_P(Q)$ is not normal in $\F$, then $\F = P\C_\F(Q)$. In this case, $Q \cap \Z(P) \leq \Z(\F)$.

 \item If $p$ is odd or $\F$ is $S_4$-free, then $\F$ is constrained.

\end{enumerate}

\end{thm}

\begin{proof}
If $Q\C_P(Q)$ is not normal in $\F$ and $\F$ is sparse, we
have $\N_\F(Q\C_P(Q)) = \F_P(P)$. By Lemma \ref{Stancu}, $\F = P\C_\F(Q)$, proving (1).
If $\Z(P)$ and $\J(P)$ are not normal in $\F$, then---as $\F$ is sparse---their $\F$-normalizers are all trivial. By \cite[Theorem 4.5]{DiazGlesserMazzaParkGT}, we conclude $\F = \F_P(P)$,
 a contradiction proving that $\O_p(\F) \neq 1$. Set $Q = \O_p(\F)$. If $Q$ is $\F$-centric, then $\F$ is constrained. Therefore, we assume that $Q$ is not $\F$-centric. This implies
that $Q$ is a proper subgroup of $Q\C_P(Q)$ and, as $Q = \O_p(F)$, we have that
$Q\C_P(Q)$ is not normal in $\F$. Applying (1), we get $\F = P\C_\F(Q)$. If $\F/Q$ is trivial, then so is $\F$ by Proposition \ref{F=PC_F(Q)2}, giving a contradiction. If $\F$ is $S_4$-free, then \cite[Proposition 6.3]{KessarLinckelmann2008} implies that $\F/Q$ is also $S_4$-free. Thus, by \cite[Theorem
4.5]{DiazGlesserMazzaParkGT}, at least one of $\C_\F(\Z(P/Q))$ and $\N_\F(\J(P/Q))$ is
not trivial, so we may assume that there is a normal subgroup $R$ of $P$ properly containing $Q$ such that $\N_{\F/Q}(R/Q)$ is not trivial. Utilizing Proposition \ref{KL3.4} again, $\N_\F(R)$ is not trivial and so, since $\F$ is sparse, $R \normal \F$, contradicting the maximality of $Q$. This proves (2).
\end{proof}

In \cite{Craven2010}, a saturated fusion system $\F$ on a finite $p$-group $P$ is called \textit{$p$-solvable} if there exists a chain of strongly $\F$-closed subgroups $1 = P_0 \leq P_1 \leq \cdots \leq P_n = P$ such that $P_i/P_{i-1} \leq \O_p(\F/P_{i-1})$ for all $1 \leq i \leq n$. When such a chain exists, the length of a minimal possible chain satisfying the above is called the \textit{$p$-length} of $\F$. It is easy to see from the above proof that a sparse fusion system is $p$-solvable with $p$-length 2 when $p$ is odd or the fusion system is $S_4$-free.

A careful reading of the proofs of Lemma \ref{Stancu} and Theorem \ref{t:sparse} show that a slightly weaker condition on $Q$ will suffice, namely we only require that $Q$ be a weakly $\F$-closed subgroup contained in every subgroup $T$ of some conjugation family for $\F$. For example, if $Q \normal \F$, then by Proposition \ref{p:O_p(F) in essentials}, $Q$ is contained in every $\F$-essential subgroup.

Note that the statement of Theorem \ref{t:sparse}(2) only considers the case where $\F$ is $S_4$-free when $p=2$, but that the result holds for $\F_{D_8}(S_4)$. This led us, in an earlier version of this paper, to conjecture that all sparse fusion systems are constrained. However, David Craven pointed out a family of counterexamples to this conjecture. Take, for example, the fusion system on $D_{16}$ afforded by $\mathrm{PGL}(2,7)$. This system is easily seen to be sparse (see \cite[Example 8.8]{Linckelmann2007} for details) and not constrained. It is still an open question as to whether there exist any sparse exotic fusion systems.

\section{$p$-Nilpotency Criterion}
\subsection*{$\Z\J$ and the Stellmacher Functor}\label{ss:ZJ and W}
We use Theorem \ref{t:sparse} to give a streamlined version of the proofs of \cite[Theorem A]{KessarLinckelmann2008} and of \cite[Theorem 1.3]{OnofreiStancu2009}. 

\begin{thm}
 Let $p$ be an odd prime and let $\F$ be a saturated fusion system on a finite $p$-group $P$. The following are equivalent:
 \begin{enumerate}
 \item $\F = \F_P(P)$
 
 \item $\N_\F(\Z(\J(P))) = \F_P(P)$
 
 \item $\N_\F(\W(P)) = \F_P(P)$.
 \end{enumerate}

Here $\J(P)$ denotes the Thompson subgroup of $P$ and and $\W$ denotes the Stellmacher functor (as in \cite{Stellmacher1996}).
\end{thm}

\begin{proof}
It is clear that (1) implies both (2) and (3). We will prove (2) implies (1) and the proof that (3) implies (1) will be the same almost verbatim (referring, of course, to Stellmacher's result instead).
Let $\F$ be a minimal counterexample with respect to the number $|\F|$
 of morphisms in $\F$. If $\CG$ is a proper subfusion system of $\F$ on $P$,
  then $\N_\CG(\Z(\J(P))) \subseteq \N_\F(\Z(\J(P))) = \F_P(P)$ and so by
   the minimality of $\F$, we have $\CG = \F_P(P)$. In
    particular, $\F$ is sparse. By Theorem \ref{t:sparse}, $\F$ is constrained and so by \cite[Proposition 4.3]{BrotoCastellanaGrodalLeviOliver2005}
       there exists a finite group $G$ with Sylow $p$-subgroup $P$ such that
        $\F = \F_P(G)$. This implies $\F_P(\N_G(\Z(\J(P)))) = \N_\F(\Z(\J(P)))
         = \F_P(P)$. By Glauberman's and Thompson's $p$-nilpotency theorem
          for groups, $\F = \F_P(G) = \F_P(P)$, a contradiction.
\end{proof}

\subsection*{A Theorem of Navarro}\label{ss:Navarro} We now generalize a result of Navarro and then translate it to fusion systems. In the following, $P'$ denotes the derived subgroup of $P$, i.e., the smallest normal subgroup of $P$ with abelian quotient, and $\Phi(P)$ denotes the Frattini subgroup of $P$, i.e., the smallest normal subgroup of $P$ with elementary abelian quotient.  

\begin{thm}[Navarro]\label{t:Navarro}
 Let $G$ be a finite group with Sylow $p$-subgroup $P$. If $\N_G(P)$ is $p$-nilpotent, then $\N_G(Q)$ is $p$-nilpotent for every subgroup $P' \leq Q \leq \Phi(P)$.
\end{thm}

The original statement of Navarro's theorem is slightly weaker, namely, it makes the stronger assumption that $\N_G(P) = P$ and only considers the case where $Q = P'$. The following proof is based on one given by I.M. Isaacs for Navarro's original statement. The inspiration for considering this stronger version is in Remark \ref{r:Navarro strong enough}.

\begin{proof}
As $P' \leq Q$, $P$ is a Sylow $p$-subgroup of $\N_G(Q)$. Also, $\N_{\N_G(Q)}(Q) = \N_G(Q)$ and so, without loss of generality, we assume that $Q \normaleq G$. As \[[P, \N_G(P)] = [P, P\C_G(P)] = P' \leq Q,\] we have $P/Q \leq \Z(\N_G(P)/Q) = \Z(\N_{G/Q}(P/Q))$ and so by Burnside's normal $p$-complement theorem (\cite[Theorem 7.4.3]{Gorenstein1980}), $P/Q$ has a normal complement in $G/Q$. Let $Q \normaleq K \normaleq G$ such that $G = PK$ and $P \cap K = Q$. The Schur--Zassenhaus theorem (\cite[Theorem 6.2.1]{Gorenstein1980}) implies that there exists $L \leq K$ such that $K = QL$ and $Q \cap L = 1$.
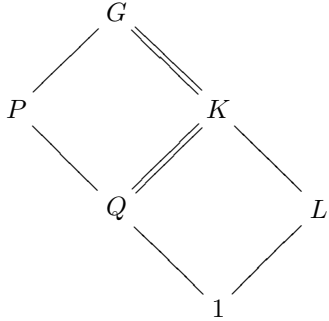
\[
 \xymatrix{& G \ar@{-}[dl] \ar@{=}[dr] & & \\ P \ar@{-}[dr] & & K \ar@{=}[dl] \ar@{-}[dr] & \\ & Q \ar@{-}[dr] & & L \ar@{-}[dl] \\ & & 1 &}
\]

 As $Q$ is solvable, any two such complements are $K$-conjugate. By the Frattini argument, $G = K \N_G(L) = QL\N_G(L) = Q\N_G(L)$. Dedekind's Lemma now implies that $P = Q\N_P(L)$ and since $Q \leq \Phi(P)$, this gives $P = \N_P(L)$. As $G = PK = PQL = PL$, we have $L \normaleq G$ and so $L$ is a normal $p$-complement for $G$, i.e., $G$ is $p$-nilpotent.

\end{proof}

\begin{prop}\label{F=PC_F(Q)}
 Let $\F$ be a fusion system on a finite $p$-group $P$ such that $\N_\F(P)$ is trivial. If $P' \leq Q \leq P$ such that $\F = P\C_\F(Q)$, then $\F$ is trivial.

\end{prop}

\begin{proof}
 By Burnside's fusion theorem (see ~\cite[Theorem 3.8]{Linckelmann2007}), $\F/Q = \N_{\F/Q}(P/Q)$ and by Proposition \ref{KL3.4}, $\N_{\F/Q}(P/Q)$ is trivial. Therefore, $\F/Q$ is trivial and, applying Proposition \ref{KL3.4} again, $\F$ is trivial. 
\end{proof}

	\begin{thm}\label{Navarrofusion}
 Let $\F$ be a fusion system on a finite $p$-group $P$. If $\N_\F(P)$ is trivial, then $\N_\F(Q)$ is trivial for every subgroup $P' \leq Q \leq \Phi(P)$.
\end{thm}

\begin{proof}
 Without loss of generality, we may assume that $Q \normal \F$.
  Let $\F$ be a minimal counterexample with respect to $|\F|$ (so that $Q > 1$).
   If $Q$ is $\F$-centric, then $\F$ is constrained. By ~\cite[Proposition 4.3]{BrotoCastellanaGrodalLeviOliver2005},
    $\F$ is the fusion system of a finite group $G$ with Sylow
     $p$-subgroup $P$ and satisfying $\N_G(P)=P\C_G(P)$. The result now follows from
       Theorem \ref{t:Navarro}. So we assume that $Q < Q\C_P(Q)$ and
        that $Q\C_P(Q) \nnormal \F$. As $\F$ is sparse and
        $\N_\F(P)$ is trivial, Theorem \ref{t:sparse} and Proposition \ref{F=PC_F(Q)} imply that $\F = P\C_\F(Q) = \F_P(P)$, a contradiction.
\end{proof}

\begin{rem}\label{r:Navarro strong enough}
  For any choice of $Q$, the proof of Theorem \ref{Navarrofusion} only requires Navarro's original theorem (where $\N_G(P) = P$) and not the full strength of Theorem \ref{t:Navarro}. In fact, if $Q$ is a normal $\F$-centric subgroup of $P$, then $\F = \F_P(G)$ for some finite group $G$ with $Q \normaleq G$, $\C_G(Q) \leq Q$ and, since $\Aut_\F(P)$ is a $p$-group, $\N_G(P) = P\C_G(P)$. Consequently, $\C_G(P) \leq \C_G(Q) = \Z(Q) \leq P$ and so $\N_G(P) = P$. This gives an excellent example of a statement about groups being used to prove a result in the context of fusion systems and thereby obtaining a stronger result about groups.
\end{rem}


\subsection*{Navarro's Theorem for Blocks}
Recall that if $b$ is a $p$-block of a finite group $G$ over an algebraically closed field and if $(P,e_P)$ is a maximal $b$-Brauer pair, then for every subgroup $Q$ of $P$, there exists a unique block $e_Q$ of $\C_G(Q)$ such that $(Q, e_Q) \leq (P,e_P)$ (for details about this inclusion, see \cite{AlperinBroue1979} or \cite{BrouePuig1980}). The group $G$ acts on the set of $b$-Brauer pairs by conjugation and this gives rise to a saturated fusion system on $P$ where for $Q,R \leq P$, conjugation by an element $g \in G$ is in the fusion system if it respects the $b$-Brauer pair structure, i.e., if $(Q,e_Q)^g \leq (R,e_R)$. The $\F$-automorphism groups of subgroups of $P$ are easily seen to be $\Aut_\F(Q) \cong \N_G(Q,e_Q)/\C_G(Q)$ and $|\Out_\F(P)|$ is called the \textit{inertial index} of $b$. The block $b$ is called \textit{nilpotent} if the corresponding saturated fusion system is trivial. For more details and a presentation of the structure of nilpotent blocks, we refer the reader to \cite{Thevenaz1995}. For an explicit proof that a block gives rise to a saturated fusion system, see \cite{Kessar2007}. The following corollary generalizes the classical result (see \cite{Kulshammer1980} or \cite{BrouePuig1980}) that a block with inertial index 1 and abelian defect group is nilpotent.

\begin{cor}\label{cor to Navarro}
 Let $b$ be a $p$-block of a finite group $G$ with inertial index 1 and maximal $b$-Brauer pair $(P, e_P)$. Let $P' \leq Q \leq \Phi(P)$ and let $e_Q$ be the unique block of $\C_G(Q)$ such that $(Q, e_Q) \leq (P, e_P)$. If $f_Q$ is a block of $\N_G(Q)$ covering $e_Q$ and which is in Brauer correspondence with $b$, then $f_Q$ is nilpotent. In particular, if $Q \normaleq G$, then $b$ is nilpotent.
 \[
 \xymatrix{G \ar@{-}[dd] & b \\
 & \\
 \N_G(Q) \ar@{-}[d] & f_Q \\
 \N_G(Q,e_Q) \ar@{-}[d] & e_Q \\
 \C_G(Q) & e_Q
 }
 \]
\end{cor}

\begin{proof}
Let $\F = \F_{(P,e_P)}(G,b)$ be the saturated fusion system on $P$ corresponding to the block $b$ and the maximal $b$-Brauer pair $(P,e_P)$. The condition that $b$ has inertial index 1 is equivalent to the condition that $\N_\F(P)$ is trivial. Therefore, by Theorem \ref{Navarrofusion}, $\F_P(P) = \N_\F(Q) = \F_{(P, e_P)}(\N_G(Q,e_Q), e_Q)$ and hence $e_Q$ is nilpotent as a block of $\N_G(Q,e_Q)$. As $Q$ is a normal $p$-subgroup of $\N_G(Q)$, $f_Q = \tr_{\N_G(Q, e_Q)}^{\N_G(Q)}(e_Q)$ is a block of $\N_G(Q)$ covering $e_Q$ and is in Brauer correspondence with $b$. Now, \cite[Proposition 2.13]{Kessar2006} implies that $\F_{(P, e_P)}(\N_G(Q,e_Q), e_Q) = \F_{(P, e_P)}(\N_G(Q), f_Q)$ completing the proof.
\end{proof}


\subsection*{Slim $p$-Groups}\label{ss:slim}

In \cite{Weigel2009pre}, Weigel makes the following definition. Set $Y_1 = C_p \wr C_p$ and for $m > 1$, define $Y_m$ to be the pull-back in the diagram:
\[
\xymatrix{C_p \wr C_p \ar[r] & C_p \\ Y_m \ar@{.>}[r] \ar@{.>}[u] & C_{p^m} \ar[u]}
\]

\begin{defn} With the notation as above, a finite $p$-group $P$ is \textit{slim} if $Y_m$ is not a subgroup of $P$ for all $m \geq 1$.
\end{defn}

\begin{thm}[\cite{Weigel2009pre}]\label{t:Weigel}
Let $G$ be a finite group and let $P$ a Sylow $p$-subgroup of $G$. If
\begin{enumerate}
\item $p$ is odd and $P$ is slim, or
\item $p = 2$ and $P$ is $D_8$-free,
\end{enumerate}
then $G$ is $p$-nilpotent if and only if $\N_G(P)$ is $p$-nilpotent.
\end{thm}

The following elementary lemma connects Theorem \ref{t:sparse} with the present context.

\begin{lem}\label{l:free}
Let $\F$ be a saturated fusion system on a finite $p$-group $P$ and let $H$ be a finite group with $p$-subgroup $Q$. If $P$ is $Q$-free, then $\F$ is $H$-free.
\end{lem}

\begin{proof}
Let $S$ be a fully $\F$-normalized, $\F$-centric subgroup of $P$ and let $G$ be the unique finite group with Sylow $p$-subgroup $\N_P(S)$ guaranteed by \cite[Proposition 4.3]{BrotoCastellanaGrodalLeviOliver2005}. If $G$ is $Q$-free, then it is $H$-free, so assume that $K/L \cong Q$ is a section of $G$. A Sylow $p$-subgroup $R$ of $K$ is conjugate to a subgroup of $\N_P(S) \leq P$ and hence is $Q$-free. On the other hand, since $K/L$ is a $p$-group, $Q \cong K/L \cong RL/L \cong R/(R \cap L)$, a contradiction.
\end{proof}

The restatement of Theorem \ref{t:Weigel} to fusion systems is clear and a minimal counterexample to such a result is easily seen to be sparse. Moreover, if $P$ is $D_8$-free, then the previous lemma shows that any saturated fusion system on $P$ is $S_4$-free. So, regardless of $p$, Theorem \ref{t:sparse} implies that the minimal counterexample is constrained.

\begin{thm}
Let $\F$ be a saturated fusion system on a finite $p$-group $P$. If \begin{enumerate}
\item $p$ is odd and $P$ is slim, or
\item $p = 2$ and $P$ is $D_8$-free,
\end{enumerate}
then $\F = \F_P(P)$ if and only if $\N_\F(P) = \F_P(P)$. 
\end{thm} 

Restricting to blocks we get the following corollary, again generalizing the result that a block with inertial index 1 and abelian defect group is nilpotent.

\begin{cor}
Let $b$ be a $p$-block of a finite group $G$ with defect group $P$ and inertial index 1. If $p$ is odd and $P$ is slim or if $p = 2$ and $P$ is $D_8$-free, then $b$ is nilpotent. 
\end{cor}

\section{Extremely sparse fusion sytems}\label{S:extremely sparse fusion systems}
In the spirit of \S \ref{S:sparse fusion systems}, we define an even more restrictive situation, namely that of an \emph{extremely sparse fusion system}.

\begin{defn} A nontrivial saturated fusion system $\F$ on a finite $p$-group $P$ is called \emph{extremely sparse} if the only proper subsystem of $\F$ on any subgroup $Q$ of $P$ is $\F_Q(Q)$.
\end{defn}
Clearly, if a fusion system is extremely sparse, then it is sparse. We will see that every extremely sparse fusion system is constrained, regardless of $p$. In fact, we offer two proofs of this result, the first because it seems Lemma \ref{l:essential lemma}(\ref{lp:essential no extend}) may have independent interest and the second because it gives a clearer picture of this situation. The author thanks David Craven and Radu Stancu for their helpful input on this section.

%

\begin{thm}\label{t:extremely sparse}
Every extremely sparse fusion system is constrained.
\end{thm}

\begin{proof}
Let $\F$ be an extremely sparse fusion system on a finite $p$-group $P$. By Lemma \ref{l:no essentials}, we may assume that $P$ has a fully $\F$-normalized, $\F$-essential subgroup $Q$. By Lemma \ref{l:essential lemma}(\ref{lp:essential no extend}), there exists $\varphi \in \Aut_\F(Q)$ such that $\N_\varphi = Q$. As any morphism in $\F_{\N_P(Q)}(\N_P(Q))$ extends to $P$, we conclude that $\varphi \notin \Aut_{\N_P(Q)}(Q)$ and hence $\N_\F(Q)$ is nontrivial. Since $\F$ is extremely sparse, it follows that $\N_\F(Q) = \F$. As $Q$ is $\F$-essential, it is $\F$-centric and so $\F$ is constrained.
\end{proof}

Now, we offer a different proof, one that provides a simple classification of all extremely sparse fusion systems.

\begin{thm}\label{t:extremely sparse classification}
Let $\F$ be a saturated fusion system on a finite $p$-group $P$. If $\F$ is extremely sparse then
$\F = \F_P(P \rtimes A)$ where $A$ is a cyclic group of order $q$ for some prime $q \neq p$.
\end{thm}

\begin{proof}
If $Q$ is an $\F$-essential subgroup of $P$, then, by Lemma \ref{l:essential lemma}(\ref{lp:p'-autos for essentials}), there exists a nontrivial $p'$-element $\alpha$ in $\Out_\F(Q)$. It follows that $\F_Q(Q)$ is a proper subsystem of $\F_Q(Q \rtimes \langle \alpha \rangle)$. However, Lemma \ref{l:essential lemma}(\ref{lp:P not essential}) implies $Q \neq P$ and so $\F_Q(Q \rtimes \langle \alpha \rangle)$ is a proper subsystem of $\F$, contradicting the extreme sparseness of $\F$. Therefore, $\rank_e(\F) = 0$ and by Lemma \ref{l:no essentials}, $\F = \F_P(P \rtimes \Out_\F(P))$. As $\F$ is nontrivial, there exists a nontrivial $q$-automorphism $\beta$ of $P$ for some prime $q \neq p$. The sparseness of $\F$ implies that $\F = \F_P(P \rtimes \langle \beta \rangle)$. 
\end{proof}

Let $\F$ be a saturated fusion system on a finite $p$-group $P$. The {\em focal subgroup} of $\F$ is defined as
\[
[P,\F] = \langle x^{-1}\varphi(x) \mid x \in P, \varphi \in \Hom_\F(\langle x \rangle, P) \rangle.
\]
As with the analogue for group theory, the focal subgroup of a fusion system controls the existence of subsystems of $p$-power index with abelian quotient. For more detail and properties we refer the reader to \cite{BrotoCastellanaGrodalLeviOliver2007} and \cite{DiazGlesserMazzaParkTr}. In the former reference, the reader can find a definition of $O^p(\F)$. Here it suffices to say that $O^p(\F)$ is the unique saturated subfusion system of $\F$ on $[P,O^p(\F)]$ with $p$-power index and that $\F$ is the unique saturated subfusion system of $\F$ on $P$ with $p$-power index.

\begin{cor} Let $\F$ be an extremely sparse fusion system on a finite $p$-group $P$ and let $A$ be a $p'$-group of $\F$-automorphisms of $P$ such that $\F = \F_P(P \rtimes A)$. 
\begin{enumerate}
\item If $Q < P$, then $\N_A(Q) = \C_A(Q)$.\label{c:nosubcentral}
\item $[P, \F] = P$.\label{c:[P,F]=P}
\item $O^p(\F) = \F$.\label{c:O^p(F)=F}
\end{enumerate}
\end{cor}

\begin{proof}
Let $1 \neq \alpha \in A$ such that $\alpha(Q) = Q$. This implies $\alpha|_Q$ is a $p'$-automorphism in $\Aut_\F(Q)$. Since $Q < P$, the subsystem $\F_Q(Q \rtimes \langle \alpha \rangle)$ is a proper subsystem of $\F$ and, hence, is trivial, an impossibility unless $\alpha$ is the identity on $Q$. This proves (\ref{c:nosubcentral}). For $u \in P$,
\[
\alpha(u^{-1}\alpha(u)) = \alpha(u)^{-1}\alpha(\alpha(u)) \in [P, \alpha],
\] 
and so $\alpha$ normalizes $[P,\alpha]$. As $\alpha$ induces the identity on $P/[P,\alpha]$, $\alpha$ cannot centralize $[P, \alpha]$ (else it would be the identity on $P$). By (\ref{c:nosubcentral}) we conclude that $P =[P, \alpha] \leq [P,\F]$, proving (\ref{c:[P,F]=P}). Finally, as $P/[P, O^p(\F)]$ controls the existence of a saturated subsystem of $\F$ with $p$-power index and $P/[P,\F]$ controls the existence of a saturated subsystem of $\F$ with $p$-power index on an abelian quotient, $[P,\F] = P$ implies $[P, O^p(\F)] = P$ and so there is no proper saturated subsystem with $p$-power index. This gives (\ref{c:O^p(F)=F}).   
\end{proof}

For a saturated fusion system $\F$ on a finite $p$-group $P$ and $Q \leq P$, set $[Q,\F; 0] = Q$ and, for positive integers $i$, define
\[
[Q, \F; i] = [[Q, \F; i-1], \F].
\]
As $[Q,\F; i] \leq [Q, \F; i-1]$ for any postive integer $i$, we may define
\[
[Q, \F; \infty] = \bigcap\limits_{i = 0}^{\infty} [Q, \F; i].
\]
Note that if $\CG$ is a subfusion system of $\F$ on $P$, then $[P, \CG; i] \leq [P, \F; i]$ for all $i$. Similarly, if $Q \normal \F$, then $[P/Q, \F; i] \leq [P, \F; i]Q/Q$ for all $i$. Our final corollary generalizes a well-known result of groups (see \cite[Theorem 4.3]{HigmanDG1953} or \cite[Proposition 12.4]{Passman1968}) to fusion systems.

\begin{cor}\label{hyperfocal}
 Let $\F$ be a saturated fusion system on a finite $p$-group $P$. If $[P, \F; \infty] = 1$, then $\F = \F_P(P)$.
\end{cor}

\begin{proof}
 Let $\F$ be a minimal counterexample with respect to $|\F|$,
  the number of morphisms in $\F$. If $\CG$ is a proper
   subfusion system of $\F$ on $Q \leq P$, then $[Q, \CG; \infty]
    \leq [P, \F; \infty] = 1$ and so by the minimality of
     $\F$, we have $\CG = \F_Q(Q)$. Therefore, $\F$ is
     extremely sparse. However, the previous corollary implies that $[P, \F; \infty] = P$. We conclude that $P = 1$ and $\F$ is trivial, a contradiction.


\end{proof}

Note that unlike many of the previous proofs, this one does not utilize the original theorem from group theory, so that this gives an alternate proof of that result.

\section{acknowledgements}
The author thanks I.M. Isaacs for several valuable discussions at MSRI about Navarro's result as well as supplying the proof of Theorem \ref{t:Navarro} in its original form. Thanks to Burkhard K\"ulshammer for suggesting Corollary \ref{cor to Navarro} and Radha Kessar for her advice and guidance in the writing of this article. The author thanks the referee for several insightful comments that improved the exposition. Finally, thank you to David Craven for coming up with the examples of non-constrained sparse systems and for helpful suggestions at several points in the paper.

\bibliography{mybib}

\end{document}